\documentclass[10pt]{article}
\usepackage{amsmath, amsthm}

\usepackage{fancyhdr}

\title {\textbf{Hall  numbers of some complete $k-$partite
graphs
}}
\date{}
\newtheorem{theorem}{Theorem}

\begin{document}

\begin{center}
\textbf{\Large Hall  numbers of some complete $k-$partite graphs}

\addvspace{\bigskipamount} Julian A. Allagan \\ School of Science Technology Engineering and Mathematics, \\
Gainesville State College,  Watkinsville, GA- 30677, USA 
aallagan@gmail.com \\
\end{center}


\begin{abstract} The Hall number is a graph parameter closely related to the choice number. Here it is shown
that the Hall numbers of the complete multipartite graphs
$K(m,2,\ldots,2)$, $m\ge 2$, are equal to their choice numbers.
\end{abstract}

\section{Introduction}
Throughout this paper, the graph $G=(V,E)$ will be a finite simple
graph with vertex set $V=V(G)$ and edge set $E=E(G)$.

 A \emph{list assignment} to the graph $G$ is a
function $L$ which assigns a finite set (list) $L(v)$ to each vertex
$v \in V(G)$.

A {\it proper $L-$coloring} of $G$ is a function $\psi: V(G)\to$
$\displaystyle \bigcup_{v\in V(G)}{L(v)}$ satisfying, for every $u$,
$v$ $\in V(G)$,
\begin{itemize}
\item [(i)] $\psi(v)$ $\in L(v)$,
\item [(ii)] $ uv \in E(G) \to\psi(v) \neq \psi(u)$.
\end{itemize}

The \emph{choice number} or {\it list$-$chromatic number of G},
denoted by $ch(G)$, is the smallest integer $k$ such that there is
always a proper $L-$coloring of $G$ if $L$ satisfies $|L(v)|\geq k$
for every $v$ $\in V(G)$. With $\chi$ denoting the chromatic number,
it is easy to see, and well known, that $\chi(G)\le ch(G)$. The
extremal equation $\chi(G)= ch(G)$ is a major research interest; see
\cite{Eno-Et}, \cite{Erd-Rub}, and \cite{Grav}.

\subsection{Hall's Theorem}

\begin{theorem}(\textnormal{P. Hall \cite{Hall}}). Suppose $A_1,A_2,\ldots,A_n$ are (not necessarily distinct) finite  sets.
There exist distinct elements $a_1, a_2,\ldots,a_n $ such that
$a_i\in A_i$, $i = 1, 2,\ldots, n$, if and only if for each $J
\subseteq \{1, 2,\ldots , n\}$, \\ $ \displaystyle | \bigcup_ {j\in
J} A_j | \ge |J|$.
\end{theorem}
A list of distinct elements $a_1,\ldots,a_n$ such that $a_i\in A_i$,
$i=1,\ldots,n$, is called a \emph{\underline{system of distinct
representatives}} of the sets $A_1,\ldots, A_n$. A proper
$L-$coloring of a complete graph $K_n$ is simply a system of
distinct representatives of the finite lists $L(v)$, $v\in V$, and
any list $A_1,A_2,\ldots,A_n$ of sets can be regarded as lists
assigned to $K_n$. Therefore, as noted in \cite{HJ}, Hall's theorem
can be restated as:

\begin{theorem}(\textnormal{Hall's theorem restated}). Suppose that $L$ is a list assignment to $K_n$. There is a proper
$L-$coloring of $K_n$ if and only if, for all $U\subseteq V(K_n)$,
$|L(U)| = \displaystyle | \bigcup_ {u\in U} L(u) | \ \ge \ |U|$.
\end{theorem}

Let $L$ be a list assignment to a simple graph $G$, $H$ a subgraph
of $G$ and $\mathcal{P}$ a set of possible colors. If $\psi: V(G)\to
\mathcal{P}$ is a proper $L-$coloring of $G$, then for any subgraph
$H\subset G$, $\psi$ restricted to $V(H)$ is a proper $L-$coloring
of $H$.

For any $\sigma \in \mathcal{P}$, let $H(\sigma, L)= \ < \{v\in
V(H)\ |\ \sigma \in L(v)\} >$ denote the subgraph of $H$ induced by
the support set $\{v\in V(H)\ |\ \sigma \in L(v)\}$. For
convenience, we
 sometimes simply write $H_\sigma $.

For each $\sigma \in \mathcal{P}$, $\psi^{-1}(\sigma)=\{v\in V(G)\
|\ \psi(v)=\sigma \}\subseteq V(G_\sigma)$; $\psi^{-1}(\sigma)$ is
an independent set because $\psi$ is a proper $L-$coloring. Further,
$\psi^{-1}(\sigma)\cap V(H) \subseteq V(H_\sigma)$. So,
$|\psi^{-1}(\sigma)\cap V(H)|\le \alpha(H_\sigma)$ where $\alpha$ is
the vertex independence number. This implies that

$\displaystyle \sum_{\sigma \in \mathcal{P}}{\alpha(H_\sigma)} \ge
\sum_{\sigma \in \mathcal{P}}|\psi^{-1}(\sigma)\cap V(H)|=|V(H)|$
for all $H\subseteq G $.

\vspace{.3in}

When $G$ and $L$ satisfy the inequality $$\sum_{\sigma \in
\mathcal{P}}{\alpha(H_\sigma)} \ge |V(H)| \ \ \ \ \ \ \ \ \ \ \ \
(3.1)$$ for each subgraph $H$ of $G$, they are said to satisfy
\emph{\textbf{Hall's condition}}. By the discussion preceding,
Hall's condition is a necessary condition for a proper $L-$coloring
of $G$. Because removing edges does not diminish the vertex
independence number, for $G$ and $L$ to satisfy Hall's condition it
suffices that (3.1) holds for all induced subgraphs $H$ of $G$.

Hall's condition is sufficient for a proper coloring when $G=K_n$,
because if $H$ is an induced subgraph of $K_n$ then for each $\sigma
\in \mathcal{P}$,
\begin{displaymath} \alpha(H_\sigma)= \left\{
\begin{array}{ll}
1 &  \ if \ \sigma \in\displaystyle\bigcup_ {v\in V(H)}L(v)\\
 0, & \ otherwise.
\end{array} \right. \end{displaymath}

 So $$\sum_{\sigma \in \mathcal{P}}{\alpha(H_\sigma)} = |\bigcup_ {v\in
V(H)}L(v)| \ ;$$ therefore Hall's condition, that $$\sum_{\sigma \in
\mathcal{P}}{\alpha(H_\sigma)}\ \ge \ |V(H)|$$ for every such $H$,
is just a restatement of the condition in Theorem 2. (It is
necessary to point out here that if $\sigma \notin L(v)$ for all
$v\in V(H)$ then $H_\sigma$ is the null graph, and
$\alpha(H_\sigma)=0$.) Consequently, Hall's theorem may be restated:
For complete graphs, Hall's condition on the graph and a list
assignment suffices for a proper coloring.

The temptation to think that there are many graphs for which Hall's
condition is sufficient can be easily dismissed. Figure $\ref{A list
assignment to K(2,2)}$ is the smallest graph with a list assignment
$L_0$ for which Hall's condition holds, and yet $G$ has no proper
$L_0-$coloring.\vspace{.1in}

\textbf{Remark.}

It is clear that if $H$ is an induced subgraph of $G$ and $H\ne G$,
then $H\subseteq G-v$ for some $v\in V(G)$. So, if $G-v$ has a
proper $L-$coloring, then $H\subseteq G-v$ must satisfy (by
necessity) (3.1). Thus, in practice, in order to show that $G$ and
$L$ satisfy Hall's condition, it suffices to verify that $G-v$ is
properly $L-$colorable for each $v\in V(G)$ and that $G$ itself
satisfies the inequality (3.1).

\vspace{.1in}
 Denoted by $h(G)$, the \textbf{\emph{Hall number}} of a graph $G$ is the
 smallest positive integer $k$ such that there is a proper $L-$coloring of $G$, whenever $G$ and $L$
 satisfy Hall's condition and $|L(v)|\ge k$ for each $v\in
 V(G)$. So, by Theorem 2, $h(K_n)=1$ for all $n$. In \cite{HJ} the following facts are shown:

\begin{itemize}
  \item[\textbf{1.}]If $|L(v)|\ge \chi(G)$ for every $v \in V(G)$ then $G$
  and $L$ satisfy Hall's Condition.
  \item[\textbf{2.}] $h(G)\le ch(G)$ for every $G$.
  \item[\textbf{3.}]If $ch(G)> \chi(G)$ then $h(G)=ch(G)$.
  \item[\textbf{4.}] If $h(G)\le \chi(G)$ then $\chi(G)=ch(G)$.
  \item[\textbf{5.}] If $H$ is an induced subgraph of $G$ then $h(H)\le
  h(G)$.
\end{itemize}

Facts 3 and 4, are essentially equivalent since $\chi$, $h \le ch$,
make $h$ a parameter of interest of study of the extremal equation $
\chi(G)=ch(G)$. These facts and the following theorems underline our
findings in the next section.

\vspace{.2in}

\textbf{Theorem A.}(Erd\"{o}s, Rubin and Taylor \cite{Erd-Rub}) Let
$G$ denote the complete $k-$partite graph $K(2,2,\ldots,2)$. Then
$ch(G)= k.$

\vspace{.2in}

\textbf{Theorem B.}(Gravier and Maffray \cite{Grav}) Let $G$ denote
the complete $k-$partite graph $K(3,3,2,\ldots,2)$. If $k>2$, then
$ch(G)=k$.

\vspace{.2in}

When $k=2$, it is shown that $ch(K(3,3))=3$. See \cite{Hof-Pet}.

\vspace{.2in}

\textbf{Corollary B.} Let $G$ denote the complete $k-$partite graph
$K(3,2,\ldots,2)$. Then $ch(G)=k$.

\begin{proof} Since $K(3,2\ldots,2)$ is a complete $k-$partite graph,\\ $k=\chi(K(3,2\ldots,2))\le ch(K(3,2\ldots,2))$.
 Further, $K(3,2\ldots,2)$ is a subgraph of the complete $k-$partite graph $K(3,3,2,\ldots,2)$.
 Therefore \\ $ch(K(3,2\ldots,2))\le k$ if
 $k>2$. Thus, $ch(K(3,2\ldots,2))=k$ if
 $k>2$. When $k=2$, we have $K(3,2)$, of which it is well known that
 the choice number is 2. See \cite{Hof-Pet}, for instance.

 \end{proof}

 \vspace{.1in}

\textbf{Theorem C.} ( Enomoto  et al. \cite{Eno-Et},2002) Let $G_k$
denote the complete $k-$partite graph $ K(4,2,\ldots,2)$. Then
\begin{displaymath} ch(G_k)= \left\{
\begin{array}{ll}
k &  \ if \ k \ is \ odd \\
k+1 & \ if \ k \ is \ even.
\end{array} \right. \end{displaymath}

\vspace{.1in}

\textbf{Theorem D.} ( Enomoto  et al. \cite{Eno-Et}) Let $G$ denote
the complete $k-$partite graph $ K(5,2,\ldots,2)$. If $k\ge 2$ then
$ch(G)=k+1$.

\vspace{.1in}

\textbf{Corollary D.} Let $G$ denote the complete $k-$partite graph
$ K(m,2,\ldots,2)$. If $k\ge 2$ and $m\ge 5$, then $h(G)=ch(G) \ge
k+1$.

 \begin{proof} Since $ch(G)\ge ch(K(5,2\ldots,2))=k+1
>k =\chi(G)$, $h(G)=ch(G)$ by the previous fact 3.

 \end{proof}

\section{Hall numbers of some complete multipartite graphs}

Throughout this section, $L$ is a list assignment to $V(G)$ such
that for each $v\in V(G)$, $L(v)\subset \mathcal{P}$, a set of
symbols. If $\sigma \notin L(v)$ for all $v\in V(G)$, then
$G_\sigma$ is the null graph. Further, we denote by $\psi$, any
attempted proper $L-$coloring of $G$.

\subsection{Example}
The following example originally appeared in \cite{HJ}. Consider the
complete bipartite graph $K(2,2)$ in Figure $\ref{A list assignment
to K(2,2)}$ with parts $V_i=\{u_i,v_i\}$, $i=1,2$ and $L_0$ the list
assignment indicated.

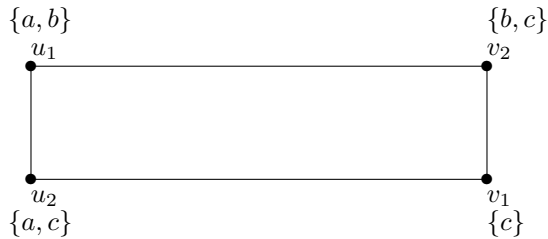
\begin{figure}[h]
  \begin{center}
\setlength{\unitlength}{1.5mm}
\begin{picture}(60,25)
\put(10,10){\circle*{1}} \put(10,20){\circle*{1}}
\put(50,10){\circle*{1}} \put(50,20){\circle*{1}}
\put(10,20){\line(1,0){40}} \put(10,10){\line(1,0){40}}
\put(10,21){$u_1$} \put(50,21){$v_2$} \put(10,8){$u_2$}
\put(50,8){$v_1$}
\put(8,23.5){$\{a,b\}$} \put(50,23.5){$\{b,c\}$}
\put(8,5.5){$\{a,c\}$} \put(50,5.5){$\{c\}$}
\put(30,10.25){\line(0,1){0}} \put(50,20){\line(0,-1){10}}
\put(10,20){\line(0,-1){10}}
\end{picture}
\end{center}
\caption{A list assignment to K(2,2).}\label{A list assignment to
K(2,2)}
\end{figure}

If $v_1$ is colored $c$, as it must be, then $u_2$ must be colored
$a$ and $v_2$ must be colored $b$ in a proper coloring, so $u_1$
cannot be properly colored.

However, we will show that $G$ and $L_0$ satisfy Hall's condition
using the argument described in a previous remark. First, for each
$v\in V(G)$, it is easy to see that $G-v$ is properly
$L_0-$colorable, meaning every proper induced subgraph $H\subset G$
satisfies, with $L_0$, the inequality (3.1) in Hall's condition. We
now proceed to verify the inequality (3.1) for $G$ itself.

Now, $\alpha(G_c)=2$ and $\alpha(G_b)=\alpha(G_a)=1$. So,
$4=\displaystyle \sum_{\sigma \in \mathcal{P}}{\alpha(G_\sigma)} \ge
|V(G)|=4$. Thus, $G$ and $L_0$ satisfy Hall's condition and yet $G$
has no proper $L_0-$coloring. Thus, $1< h(G)\le 2$ by Fact 2 and
Theorem A. Therefore, $h(G)=2$.

\subsection {Some Hall numbers }

\begin{theorem}
$h(K(2,\ldots,2))= k$ when $k\ge 2$.
\end{theorem}

 \begin{proof}

Let the partite sets of the complete $k-$partite graph $G=
K(2,\ldots,2)$ be  $V_1,\ldots , V_k$ with $V_i=\{u_i , v_i\}$, for
$i=1,2,\ldots,k$.

In Example 2.1, we showed that $h(G)=k$ when $k=2$. So, to complete
the proof, we suppose $k\ge 3$.

Let $A$ be a nonempty set of colors with $|A|=k-2$ and $a,b,c$ be
distinct colors not in A. We define $L$ a list assignment to $G$ as
follows:

\begin{itemize}
  \item[1.]$L(u_1)=A\cup\{a,b\}$, $L(u_2)=L(u_3)=\ldots =L(u_{k-1})= A\cup\{a\}$, $L(u_k)=A\cup\{c\}$
   and
  \item[2.] $L(v_1)=A\cup\{b,c\}$, $L(v_2)=L(v_3)=\ldots =L(v_{k})= A\cup\{b\}$.
\end{itemize}
Observe that $|L(v)|\ge k-1$ for every $v\in V(G)$.

\vspace{.2in}

 \textbf{Claim 1.} \emph{
 The graph $G$ is not properly $L-$colorable.}

\vspace{.2in}

\textbf{ Proof.}

In the following cases, we consider all possible distinct ways to
properly color the vertices of some part of $G$, say $V_1$ . We then
conclude that the remaining subgraph $H= G-V_1$ is not proper
$L'$-colorable where $L'=L-\{\alpha_1, \alpha_2\}$, $\displaystyle
\{\alpha_1, \alpha_2\} \in \bigcup _{v\in V_1}L(v)$. ($\alpha_1,
\alpha_2$ are not necessarily distinct colors; they are the colors
on $V_1$.) Let $\psi$ denote the attempted proper coloring.

 \textbf{Case 1:} $\psi(u_1)=b$ or $\psi(v_1)=b$.

Let $S=<\{v_2, \ldots, v_k\}>$, an induced subgraph of $H$. Then
$k-2=|A|=\displaystyle |\bigcup_{v \in V(S)}L'(v)|< |V(S)|=k-1$.
Since the subgraph $S$ is a clique, we cannot properly color $S$
from $L'$.

\textbf{Case 2:} $\psi(u_1)=a$ and $\psi(v_1)=c$.

Similarly as described in case 1, by letting $S=<\{u_2, \ldots,
u_k\}>$, it's clear that we cannot properly color $S$, from $L'$.

\textbf{Case 3:} $\psi(u_1)=\gamma$ or $\psi(v_1)=\gamma$ for some
color $\gamma \in A$.

With $S$ as in case 1, $k-2=\displaystyle |\bigcup_{v \in
V(S)}L'(v)|< |V(S)|=k-1$. Hence we cannot properly color  $H$ from
$L'$.

\vspace{.2in}

\textbf{Claim 2.} \emph{ $ \displaystyle \sum_{\sigma\in
\mathcal{P}}\alpha(G_\sigma)\ge |V(G)|$.}

\vspace{.2in}

\textbf{Proof.}

 It is clear that $\alpha(G_a)=\alpha(G_c)=1, \alpha(G_b)=2$;
further, $\alpha(G_\sigma)=2$ for every $\sigma \in A$. Hence $
\displaystyle \sum_{\sigma\in \mathcal{P}}\alpha(G_\sigma)= 2(k-2) +
4= 2k = |V(G)|$.

\vspace{.2in}

\textbf{Claim 3.} \emph{Every proper induced subgraph $H$ of $G$ is
properly $L-$colorable.}

\vspace{.2in}

 \textbf{Proof.}

 In the following cases we provide a (not necessarily unique)
proper coloring for each induced subgraph $H$ of $G$ of the form
$G-v$, $v\in V(G)$.

\textbf{Case $1$}: $H=G-u_1$.

Let $\psi(v_1)=c$ and color the $2(k-2)$ vertices of the subgraph
$G-(V_1\cup V_2)$ with the colors from $A$ (by coloring vertices of
the same part with the same color). Then let $\psi(u_2)=a$ and $
\psi(v_2)=b$.

 \textbf{Case $2$}: $H=G-v_1$.

Let $\psi(u_1)=a$ and color the $2(k-2)$ vertices of the subgraph
$G-(V_1\cup V_k)$ with the colors from $A$ with the same color
appearing on $u_i$ and $v_i$, $i=2,\ldots, k-1$. Then, let
$\psi(u_k)=c$ and $ \psi(v_k)=b$.

\textbf{Case $3$}: $H=G-u_i$, for some $2\le i \le k$.

Let $\psi(v_i)=b$ and color the remaining $2(k-2)$ vertices of the
subgraph $G-(V_i\cup V_1) $ with the colors from $A$. Then, let
$\psi(u_1)=a$ and $ \psi(v_1)=c$.

\textbf{Case $4$}: $H=G-v_i$, for some $2\le i \le k-1$.

Let $\psi(u_i)=a$ and color the remaining $2(k-2)$ vertices of the
subgraph $G-(V_i\cup V_1)$ with the colors from $A$. Then, let
$\psi(u_1)= \psi(v_1)=b$.

\textbf{Case $5$}: $H= G-v_k $.

Let $\psi(u_k)=c$ and color the $2(k-2)$ vertices of the subgraph
$G-(V_1\cup V_k)$ with the colors from $A$. Finally, let $\psi(u_1)=
\psi(v_1)=b$.

From the previous claims, we can conclude that $h(G)>k-1$. Thus, by
Theorem A and Fact 2, $h(G)=k$. This concludes the proof.
 \end{proof}

\vspace{.2in}
 \textbf{Corollary 3:} $h(K(3,2\ldots,2))=k=h(K(3,3,2\ldots,2))$ for $k>2$.
 \begin{proof}
 From Theorem 3, fact 5 and Theorem B, $k=h(K(2,2\ldots,2))\le \\
 h(K(3,2\ldots,2))\le h(K(3,3,2\ldots,2))\le ch(K(3,3,2\ldots,2))=k$. Thus,
$h(K(3,2\ldots,2))=k=h(K(3,3,2\ldots,2))$.
 \end{proof}
 We note that when $k=2$, $h(K(3,2))=2$ since $2=h(K(2,2))\le h(K(3,2))\le ch(K(3,2))=2$ by Corollary B.
Also,  since $ch(K(3,3))=3$ by \cite{Hof-Pet}, it is clear from Fact
3 that $h(K(3,3))=3$.

\begin{theorem}
Let $G$ denote the complete $k-$partite graph

$ K(4,2,\ldots,2)$ with $k\ge 2$. Then
\begin{displaymath} h(G)= \left\{
\begin{array}{ll}
k &  \ if \ k \ is \ odd \\
k+1 & \ if \ k \ is \ even.
\end{array} \right. \end{displaymath}
\end{theorem}

 \begin{proof}

 When $k$ is even, from Theorem B we have that $k=\chi(G)<ch(G)=k+1$. Thus, from Fact 3, it
is clear that $h(G)=ch(G)=k+1$ for all even $k\ge 2$.

Suppose $k\ge 3$ is odd.

Let the partite sets, or parts, $V_1,V_2,\ldots,V_k$ of the complete
$k-$partite graph $G$ be $V_1=\{x_1,x_2,x_3,x_4\}$ and $V_i=\{u_i ,
v_i\}$, $i=2,\ldots,k$, $k\ge 2$.

Let $C_1$ and $C_2$ be disjoint $k-2$ sets of colors and $0$ an
object not in $C_1\cup C_2$. Let $A=C_1\cup \{0\}$, $B=C_2\cup
\{0\}$. Let $A_1$, $A_2$ and $B_1$, $B_2$ be disjoint $(k-1)/2$ sets
of colors partitioning $A$ and $B$ respectively, and let $0\in
A_2\cap B_2$. Let $a$,$b$ be distinct objects not in $A\cup B$.
Define a list assignment $L$ to $G$ as follows:

\begin{itemize}
  \item[1.]$L(u_2)=A$, $L(v_2)=B$,
  $L(u_i)=C_1\cup \{a\}$ and $L(v_i)=C_2\cup \{b\}$,  for every  $3\leq i\leq k$  and
  \item[2.] $L(x_1)= A_1\cup B_1$, $L (x_2)= A_1 \cup B_2$, $L(x_3)= A_2\cup B_1$
   and $L(x_4)= A_2\cup B_2\cup \{a\}$
\end{itemize}

 Notice that $|L(v)|=k-1$ for every $v\in V(G)$.

 \vspace{.2in}

\textbf{Claim 1.} \emph{G is not properly $L-$colorable.}

\vspace{.2in}

 \textbf{Proof.}

 Every proper $L-$coloring of $G-V_1$ = $K (2,\ldots,
2)$ uses $k-1$ elements of $C_1\cup \{0,a\}$ and $k-1$ elements of
$C_2\cup \{0,b\}$. We proceed by exhausting the possible cases in
attempts to properly $L-$color $G$.

 \textbf{Case $1$}: suppose $\psi(u_2)\ne 0 \ne \psi(v_2)$.
 Then all  of the colors of $C_1\cup C_2\cup \{a,b\}$ will be
 used to color $G-V_1$. Hence we cannot
 color $x_1$ (since $A_1\cup B_1 \subset C_1\cup C_2$).

\textbf{Case $2$}: suppose $\psi(u_2)=\psi(v_2)=0$

\textbf{Case $2.1$}: $\psi(u_i)\ne a$ and $\psi(v_i)\ne b$ for every
$3\leq i\leq k$.

 Then all  of the colors of $C_1\cup C_2$ will be
 used to color $G-(V_1\cup V_2)$. Once again we cannot
 color $x_1$.

 \textbf{Case $2.2$}: $\psi(u_i)= a$ and $\psi(v_j)= b$ for some $i,j\ne 2$.

 Then there remains exactly one color, say $c_1\in C_1$ and exactly one color, say $c_2\in
 C_2$. If $c_1\in A_1$ and $c_2\in B_1$, then we cannot color $x_4$.
 Likewise if $c_1\in A_1$ and $c_2\in B_2$, then we cannot color
 $x_3$. Also if $c_1\in A_2$ and $c_2\in B_1$, $x_2$ cannot be
 colored and if $c_1\in A_2$, $c_2\in B_2$, $x_1$ cannot be colored.

 \textbf{Case $2.3$}: $\psi(u_i)\ne a$ for all $i\ne 2$ and $\psi(v_j)= b$ for some $j\ge 3$.
 Then there remains exactly  one color, say $c_2\in C_2$ and none of $
 C_1$. As in the previous case, if $c_2\in B_1$, then we cannot color $x_2$.
 Likewise if $c_2\in B_2$, then we cannot color either of $x_1$ and $x_3$.

 \textbf{Case $2.4$}: $\psi(u_i)= a$ for some $i\ge 3$ and $\psi(v_j)\ne b$ for all $j\ge 3$.
 Then there remains exactly  one color, say $c_1\in C_1$ and none of $
 C_2$. As before, if $c_1\in A_1$, then we cannot color either of $x_3$ and $x_4$.
 Likewise if $c_1\in A_2$, then we cannot color either of $x_1$ and $x_2$.

\textbf{Case $2.5$}: $\psi(u_i)\ne a$ and $\psi(v_j)\ne b$ for all
$3\le i,j\le k$. Clearly the coloring cannot be properly extended to
any of $x_1, x_2, x_3$.

 Notice that we can skip the case
where $\psi(u_2)= 0$ and $\psi(v_2)\ne 0$ (or vice versa), since if
there is a proper $L-$coloring with one of $u_2$, $v_2$ colored with
$0$, then there is a proper $L-$coloring with both colored $0$.

 From the previous cases we can conclude that $G$ is not properly $L-$
 colorable.

\vspace{.2in}

 \textbf{Claim 2.} \emph{$ \displaystyle
\sum_{\sigma\in \mathcal{P}}\alpha(G_\sigma)\ge |V(G)|$.}

\textbf{Proof.}

 Notice that $\alpha(G_\sigma)=2$ for every $\sigma \in
C_1\cup C_2$. Also $\alpha(G_0)=3$ and $\alpha(G_a)=\alpha(G_b)=1$.
Hence $\displaystyle \sum_{\sigma\in
\mathcal{P}}\alpha(G_\sigma)=2(2(k-2))+5 = 4k-3\ge 2k+2 = |V(G)|$
for every $k\ge 3$.

\vspace{.2in}

\textbf{Claim 3.} \emph{If $k\ge 5$, then every proper induced
subgraph $H$ of $G$ is properly $L-$colorable.}

\vspace{.2in}

\textbf{Proof.}

 We proceed by considering the possible subgraphs of $G$ obtained by deleting a single vertex.

\textbf{Case $1$}: $H=G- u_i$, for some $i$.

Let $\psi(x_2)= \psi(x_3)=\psi(x_4)=0$. Color $ G-V_1$ with the
colors from $C_1\cup C_2\cup \{a,b\}$ (colors $a, b$ included).
Hence there remains exactly one unused color of $C_1$, say $c_1$,
and arrange that $c_1\in A_1$. Let $\psi(x_1)=c_1$.

\textbf{Case $2$}: $H=G- v_i$, for some $i$. Following the coloring
argument in the previous case, there remains exactly one unused
color of $C_2$, say $c_2$, and arrange that $c_2\in B_1$. Let
$\psi(x_1)=c_2$.

\textbf{Case $3$}: $H=G- x_1$. Let $\psi(x_2)=
\psi(x_3)=\psi(x_4)=0$. It is easy to see that we can color the
remaining subgraph $ G- V_1$ with the colors from $C_1\cup C_2\cup
\{a,b\}$ ($a, b$ included).

\textbf{Case $4$}: $H=G- x_2$. Let $\psi(u_2)= \psi(v_2)=0$, and
$\psi(x_4)=a$. Color the vertices of $ G-(V_1\cup V_2)$ with the
colors from $C_1\cup C_2\cup \{b\}$ ($b$ included). Then there
remains exactly one unused color of $C_2$, say $c_2$, and arrange
that $c_2\in B_1$. Let $\psi(x_1)= \psi(x_3)=c_2$.

\textbf{Case $5$}: $H=G- x_4$. Let $\psi(u_2)= \psi(v_2)=0$. Color
the vertices of $ G-(V_1\cup V_2)$ with the colors from $C_1\cup
C_2\cup \{a,b\}$ ($a,b$ included). Then there remains exactly one
unused color of $C_1$, say $c_1$, and arrange that $c_1\in A_1$, and
exactly one unused color of $C_2$, say $c_2$, and arrange that
$c_2\in B_1$. Let $\psi(x_1)=c_1=\psi(x_2)$ and $ \psi(x_3)=c_2$.

\textbf{Case $6$}: $H=G- x_3$. Let $\psi(u_2)= \psi(v_2)=0$. Color
the vertices of $ G-(V_1\cup V_2)$ with the colors from $C_1\cup
C_2\cup \{a,b\}$ ($a,b$ included). Then there remains exactly one
unused color of $C_1$, say $c_1$, and arrange that $c_1\in A_1$, and
exactly one unused color of $C_2$, say $c_2$, and arrange that
$c_2\in B_2$. Let $\psi(x_1)=c_1=\psi(x_2)$ and $\psi(x_4)=c_2$.

Notice here that when $k=3$, $A_2= B_2=\{0\}$. Therefore, the
attempted coloring of $H=G-x_3$ in case 6 fails, and, in fact $H$ is
not properly $L-$ colorable. However, $H=G-x_3$ with the given list
assignment $L$ satisfies the inequality (3.1). We can safely end the
proof here when $k=3$.

Still, there follows a list assignment specifically for the case
when $k=3$, which we hope will be of interest.

We define a list assignment $L$  to $G=K(4,2,2)$ as follows:

\begin{itemize}
  \item[1.]$L(u_2)=\{1,0\}$, $L(v_2)=\{2,0,c\}$, $L(u_3)=\{1,a\}$, $L(v_3)=\{2,b\}$
   and
  \item[2.] $L(x_1)=\{1,2\}$, $L(x_2)=\{1,0\}$, $L(x_3)=\{0,a\}$ and $L(x_4)=\{b,c\}$
\end{itemize}
It is easy to verify that $G$ and $L$ satisfy the previous claims
$1$ and $2$. We proceed therefore to verify only claim $3$ for the
subgraphs $H$ of $K(4,2,2)$ in the following cases.

\vspace{.1in}

\textbf{Case$1$}: $H=G-u_2$.

Let $\psi(v_2)=2, \psi(u_3)=a, \psi(v_3)=b$. Also
$\psi(x_2)=0=\psi(x_3), \psi(x_1)=1$ and $\psi(x_4)=c$.

\textbf{Case$2$}: $H=G- v_2$.

Let $\psi(u_2)=1, \psi(u_3)=a, \psi(v_3)=b$. Also
$\psi(x_2)=0=\psi(x_3), \psi(x_1)=2$ and $\ \ \  \ \ \ \ \psi(x_4)=c
$.

\textbf{Case$3$}: $H=G- u_3$.

Let $\psi(u_2)=\psi(v_2)=0, \psi(v_3)=b$. Also
$\psi(x_1)=1=\psi(x_2), \psi(x_3)=a$ and $\psi(x_4)=c$

\textbf{Case$4$}: $H=G- v_3$.

Let $\psi(u_2)=1, \psi(v_2)=c, \psi(u_3)=a$. Also $ \psi(x_1)=2,
\psi(x_2)=0=\psi(x_3)$ and $\ \ \ \ \    \ \psi(x_4)=b $.

\textbf{Case$5$}: $H=G- x_1$.

Let $\psi(u_2)=1$, $\psi(v_2)=2$, $\psi(u_3)=a$ and$\psi(v_3)=b$ .
Also let $ \psi(x_1)=0=\psi(x_2)$ and $\psi(x_4)=c $.

\textbf{Case$6$}: $H=G- x_2$.

Let $\psi(u_2)=0= \psi(v_2), \psi(u_3)=1$ and$\psi(v_3)=b$ . Also $
\psi(x_1)=2,\psi(x_3)=a$ and $\psi(x_4)=c $.

\textbf{Case$7$}: $H=G- x_3$.

Let $\psi(u_2)=0= \psi(v_2), \psi(u_3)=a$ and$\psi(v_3)=b$ . Also $
\psi(x_1)=1=\psi(x_2)$ and $\psi(x_4)=c $.

 \textbf{Case$8$}: $H=G- x_4$.

Let $\psi(u_2)=1,\psi(v_2)=c,\psi(u_3)=a$ and $\psi(v_3)=b$ . Also
$\psi(x_1)=2 $ and  $ \psi(x_2)=0=\psi(x_3)$.


We conclude that $G$ and $L$ satisfy Hall's Condition. So, $k\le
h(G)\le ch(G)=k$ by Fact 2 and Theorem B. Therefore, $h(G)=k$ for
all $k\ge 3$ odd.

 \end{proof}

\vspace{.1in}
 \textbf{Corollary 4:} For $m\ge 2$, $k\ge
2$, $h(K(m,2\ldots,2))=ch(K(m,2\ldots,2))$.

\vspace{.1in}

 \begin{proof} This follows from Corollaries D and 3, and
Theorems C, D, 3 and 4.
 \end{proof}

\vspace{.1in}

 \textbf{Conjecture:} If $G$ is a complete multipartite
graph with all parts of size greater than 1, then $h(G)=ch(G)$.

\vspace{.1in}

 Since $h(K_n)=1 < n =ch(K_n)$, the conclusion of the
conjecture fails if parts of size 1 are allowed. Since $h(G)=ch(G)$
whenever $\chi(G)<ch(G)$, and since $\chi(G)<ch(G)$ for "most"
complete multipartite graphs $G$ with part sizes greater than 1,
with Theorems 3 and 4 we may be within shouting distance of
confirming the conjecture.


\bigskip
\vspace{.3in}

 \centerline{\bf Acknowledgement}
\bigskip
The author expresses his gratitude to Professor Peter Johnson Jr for
communicating this problem and encouraging this work.

\end{document}